\documentclass[reqno, oneside, 12pt]{amsart}

\usepackage[letterpaper]{geometry}
\usepackage{appendix}
\usepackage{amsmath}
\usepackage{amssymb}
\usepackage{enumerate}
\usepackage{verbatim}
\usepackage{mathrsfs}
\usepackage{mathtools}
\usepackage{hyperref}
\usepackage{amsthm}
\usepackage{graphicx}

\usepackage{tikz-cd}
\newcommand*\circled[1]{\tikz[baseline=(char.base)]{
            \node[shape=circle,draw,inner sep=2pt] (char) {#1};}}

\newtheorem{theorem}{Theorem}[section]
\newtheorem{example}[theorem]{Example}

\theoremstyle{definition}

\theoremstyle{remark}
\newtheorem{remark}{Remark}

\author{Heidi Goodson}
\address{Department of Mathematics, Brooklyn College; 2900 Bedford Avenue, Brooklyn, NY 11210 USA}
\email{heidi.goodson@brooklyn.cuny.edu}

\title[An Identity for Vertically Aligned Entries in Pascal's Triangle]{An Identity for Vertically Aligned Entries in Pascal's Triangle}

\begin{document}

\begin{abstract}
The classic way to write down Pascal's triangle leads to entries in alternating rows being vertically aligned. In this paper, we prove a linear dependence on vertically aligned entries in Pascal's triangle. Furthermore, we give an application of this dependence to morphisms between hyperelliptic curves.
\end{abstract}

\maketitle

\section{Introduction}

We consider entries in the $n$th row of Pascal's triangle, where $n$ is any nonnegative integer. It is well known that the $i$th entry in this row can be computed as $\binom{n}{i}$, where $0\leq i\leq n$. For example, the 3rd entry in row 11 is $\binom{11}{3}=\frac{11\cdot10\cdot9}{3!}=165$. Figure \ref{fig:pascal} shows rows 0 through 12 of Pascal's triangle.

\begin{center}
\begin{figure}[h]
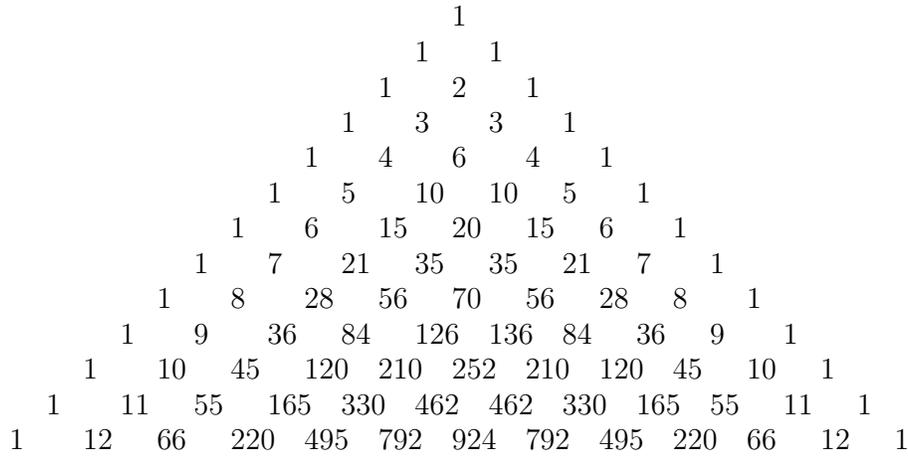

\resizebox{.75\textwidth}{!}{
\begin{tabular}{p{.5cm} p{.5cm} p{.5cm} p{.5cm} p{.5cm} p{.5cm} p{.5cm} p{.5cm} p{.5cm} p{.5cm} p{.5cm} p{.5cm} p{.5cm} p{.5cm} p{.5cm} p{.5cm} p{.5cm} p{.5cm} p{.5cm} p{.5cm} p{.5cm} p{.5cm} p{.5cm} p{.5cm} p{.5cm} }
&&&&&&&&&&&&1\\
&&&&&&&&&&&1&&1\\
&&&&&&&&&&1&&2&&1\\
&&&&&&&&&1&&3&&3&&1\\
&&&&&&&&1&&4&&6&&4&&1\\
&&&&&&&1&&5&&10&&10&&5&&1\\
&&&&&&1&&6&&15&&20&&15&&6&&1\\
&&&&&1&&7&&21&&35&&35&&21&&7&&1\\
&&&&1&&8&&28&&56&&70&&56&&28&&8&&1\\
&&&1&&9&&36&&84&&126&&136&&84&&36&&9&&1\\
&&1&&10&&45&&120&&210&&252&&210&&120&&45&&10&&1\\
&1&&11&&55&&165&&330&&462&&462&&330&&165&&55&&11&&1\\
1&&12&&66&&220&&495&&792&&924&&792&&495&&220&&66&&12&&1\\
\end{tabular}}
\vspace{.1in}
\caption[Table caption text]{Pascal's triangle.}
\label{fig:pascal}
\end{figure}    
\end{center}

Notice that entries in alternating rows are vertically aligned. For example, in Figure \ref{fig:pascal11_3} below we have circled the entries that are vertically aligned with the 3rd entry in the 11th row.  In Figure \ref{fig:pascal12_6} we have circled the entries that are vertically aligned with the 6th entry in the 12th row.
\newpage

\begin{center}
\begin{figure}[h]
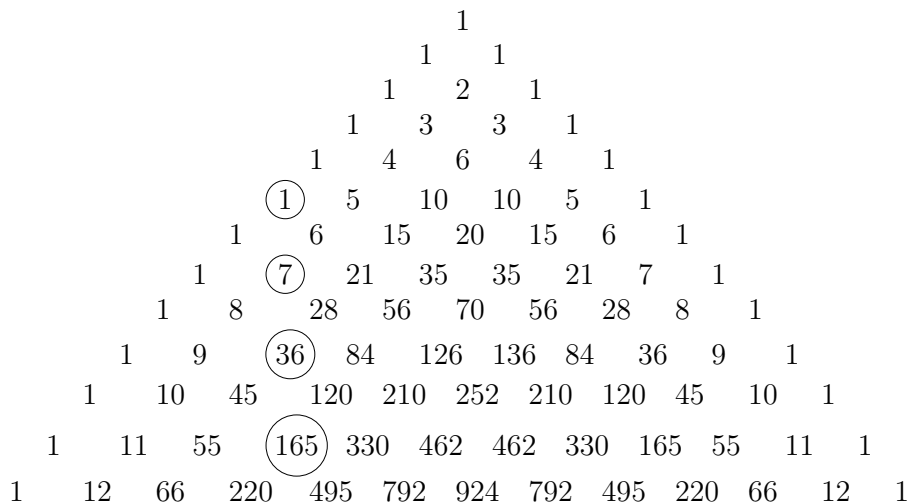

\resizebox{.75\textwidth}{!}{
\begin{tabular}{p{.5cm} p{.5cm} p{.5cm} p{.5cm} p{.5cm} p{.5cm} p{.5cm} p{.6cm} p{.5cm} p{.5cm} p{.5cm} p{.5cm} p{.5cm} p{.5cm} p{.5cm} p{.5cm} p{.5cm} p{.5cm} p{.5cm} p{.5cm} p{.5cm} p{.5cm} p{.5cm} p{.5cm} p{.5cm} }
&&&&&&&&&&&&1\\
&&&&&&&&&&&1&&1\\
&&&&&&&&&&1&&2&&1\\
&&&&&&&&&1&&3&&3&&1\\
&&&&&&&&1&&4&&6&&4&&1\\
&&&&&&&\circled{1}&&5&&10&&10&&5&&1\\
&&&&&&1&&6&&15&&20&&15&&6&&1\\
&&&&&1&&\circled{7}&&21&&35&&35&&21&&7&&1\\
&&&&1&&8&&28&&56&&70&&56&&28&&8&&1\\
&&&1&&9&&\circled{36}&&84&&126&&136&&84&&36&&9&&1\\
&&1&&10&&45&&120&&210&&252&&210&&120&&45&&10&&1\\
&1&&11&&55&&\circled{165}&&330&&462&&462&&330&&165&&55&&11&&1\\
1&&12&&66&&220&&495&&792&&924&&792&&495&&220&&66&&12&&1\\
\end{tabular}}
\vspace{.1in}
\caption{Entries vertically aligned with the 3rd entry in the 11th row.}
\label{fig:pascal11_3}
\end{figure}    
\end{center}

\begin{center}
\begin{figure}[h]
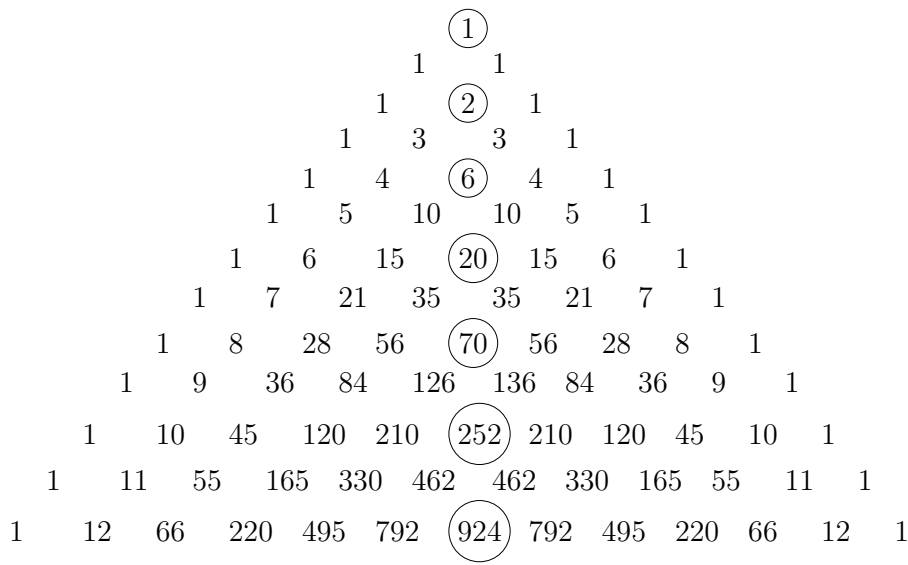

\resizebox{.75\textwidth}{!}{
\begin{tabular}{p{.5cm} p{.5cm} p{.5cm} p{.5cm} p{.5cm} p{.5cm} p{.5cm} p{.5cm} p{.5cm} p{.5cm} p{.5cm} p{.5cm} p{.6cm} p{.5cm} p{.5cm} p{.5cm} p{.5cm} p{.5cm} p{.5cm} p{.5cm} p{.5cm} p{.5cm} p{.5cm} p{.5cm} p{.5cm} }
&&&&&&&&&&&&\circled{1}&&&&&&&&&&&&\\
&&&&&&&&&&&1&&1\\
&&&&&&&&&&1&&\circled{2}&&1\\
&&&&&&&&&1&&3&&3&&1\\
&&&&&&&&1&&4&&\circled{6}&&4&&1\\
&&&&&&&{1}&&5&&10&&10&&5&&1\\
&&&&&&1&&6&&15&&\circled{20}&&15&&6&&1\\
&&&&&1&&{7}&&21&&35&&35&&21&&7&&1\\
&&&&1&&8&&28&&56&&\circled{70}&&56&&28&&8&&1\\
&&&1&&9&&{36}&&84&&126&&136&&84&&36&&9&&1\\
&&1&&10&&45&&120&&210&&\circled{252}&&210&&120&&45&&10&&1\\
&1&&11&&55&&{165}&&330&&462&&462&&330&&165&&55&&11&&1\\
1&&12&&66&&220&&495&&792&&\circled{924}&&792&&495&&220&&66&&12&&1\\
\end{tabular}}
\vspace{.1in}
\caption{Entries vertically aligned with the 6th entry in the 12th row.}
\label{fig:pascal12_6}
\end{figure}    
\end{center}

We can describe these entries in the following way. Starting with the $i$th entry in the $n$th row, i.e. $\binom{n}{i}$, the entries that are vertically aligned with this entry and above it are all of the form $$\binom{n-2k}{i-k}, $$
where $1\leq k\leq i$ and $k\leq \lfloor\frac{n}{2}\rfloor$.

For example, when $n=11$ and $i=3$, the entries that are above $\binom{11}{3}$ and vertically aligned with it  are
$$\binom92, \;\binom71,\; \binom50. $$
Observe that

$$\binom{11}{3}-11\binom92 +44\binom71-77\binom50 = 165-11\cdot36+44\cdot7-77\cdot1=0.$$

When $n=12$ and $i=6$, we have

\begin{align*}
\binom{12}{6} - 12\binom{10}{5} + 54\binom84 &- 112\binom63 + 105\binom42 - 36\binom21 + 2\binom00\\&=924 -12\cdot252+54\cdot70-112\cdot20+105\cdot6-36\cdot 2+2\cdot1\\&= 0.     
\end{align*}

In the next section, we prove a general formula for the linear dependence on vertically aligned entries in Pascal's triangle.

\section{General Formula}

\begin{theorem}\label{thm:id}
Let $n$ be a nonnegative integer and $0<i<n$. Then 

$$\sum_{k=0}^i (-1)^k\frac{n}{n-k}\binom{n-k}{k}\binom{n-2k}{i-k}=0. $$
\end{theorem}
\begin{remark}
Note that the $k=0$ term is simply $\binom{n}{i}$. If $i>n/2$, there will be some values of $k$ for which $n-2k<i-k$. For example, if $n=11$ and $i=8$, then $k=4$ has $n-2k=3<4=i-k$.   But recall that
$$\binom{m}{r}=0$$
whenever $0\leq m<r$ (see, for example, \cite[Section 1.9]{Koshy2014}). Thus, terms for which  $0\leq n-2k<i-k$ do not contribute to the sum in Theorem \ref{thm:id}. 

If $n-2k<0$, then $\binom{n-2k}{i-k}$ is no longer 0. However in this case, we have $n<2k$, which implies, $n-k<k$. Thus, $\binom{n-k}{k} =0$ instead.

Hence, all terms for which $i>n/2$ do not contribute to the sum in Theorem \ref{thm:id}. 
\end{remark}
\begin{remark}\label{rem:oeis}
The expressions $\frac{n}{n-k}\binom{n-k}{k}$ that appear in Theorem \ref{thm:id} are referred to as the Triangle of coefficients of Lucas (or Cardan) polynomials, denoted $T(n,k)$, in the On-Line Encyclopedia of Integer Sequences \cite{oeis}.
\end{remark}
\begin{proof}[Proof of Theorem \ref{thm:id}]
The following proof starts with  an identity attributed to E.H. Lockwood. For any $n\geq1$,
$$x^n+y^n = \sum^{\lfloor n/2\rfloor}_{k=0} (-1)^{k}\frac{n}{n-k}\binom{n-k}{k}(xy)^k(x+y)^{n-2k}$$
 (see, for example, \cite[Section 9.8]{Koshy2014}).
 
We separate the $k=0$ term from the summation to get
\begin{equation}\label{eqn:lockwood}
x^n+y^n = (x+y)^n+ \sum^{\lfloor n/2\rfloor}_{k=1} (-1)^{k}\frac{n}{n-k}\binom{n-k}{k}(xy)^k(x+y)^{n-2k}.    
\end{equation}

The Binomial Theorem tells us that 
\begin{equation}\label{eqn:binomial}
    (x+y)^n =\sum_{i=0}^n\binom{n}{i} x^{n-i}y^i=x^n+y^n+\sum_{i=1}^{n-1}\binom{n}{i} x^{n-i}y^i
\end{equation}

Substituting Equation \ref{eqn:binomial} into Equation \ref{eqn:lockwood} yields
\begin{equation*}
    x^n+y^n = x^n+y^n+\sum_{i=1}^{n-1}\binom{n}{i} x^{n-i}y^i+ \sum^{\lfloor n/2\rfloor}_{k=1} (-1)^{k}\frac{n}{n-k}\binom{n-k}{k}(xy)^k(x+y)^{n-2k}.
\end{equation*}
Hence,
\begin{equation}\label{eqn:coeff}
   \sum_{i=1}^{n-1}\binom{n}{i} x^{n-i}y^i+ \sum^{\lfloor n/2\rfloor}_{k=1} (-1)^{k}\frac{n}{n-k}\binom{n-k}{k}(xy)^k(x+y)^{n-2k}=0.
\end{equation}

Thus, when combining the two sums, the coefficient of each $x^{n-i}y^i$ term must equal 0. We expand the second summand in order to identify all terms of the form $x^{n-i}y^i$. The Binomial Theorem tells us that, for each $k$, 
$$(x+y)^{n-2k} = \sum_{j=0}^{n-2k} \binom{n-2k}{j} x^{n-2k-j}y^j.$$
Hence, 
\begin{equation}\label{eqn:xy}
(xy)^k(x+y)^{n-2k} = \sum_{j=0}^{n-2k} \binom{n-2k}{j} x^{n-k-j}y^{j+k}.    
\end{equation}
The values of $j$ that yield $x^{n-i}y^i$ terms are $j=i-k$. Note that we must have $k\leq i$, since otherwise $j\leq0$. Thus, the coefficient of $x^{n-i}y^i$ in Equation \ref{eqn:xy} is
$$\sum_{k=1}^i \binom{n-2k}{i-k}.$$

Hence, the sum of the coefficients of the $x^{n-i}y^i$ terms in Equation \ref{eqn:coeff} is
$$\sum_{k=0}^i (-1)^k\frac{n}{n-k}\binom{n-k}{k}\binom{n-2k}{i-k}=0,$$
where the $k=0$ term is $\binom{n}{i}$, which comes from the first summation in Equation \ref{eqn:coeff}.

\end{proof}

\section{Application to Hyperelliptic Curves}
In this section we give an application of the identity in Theorem \ref{thm:id}. Work on this application in \cite[Section 5.1]{EmoryGoodsonPeyrot} is what led the author to discover the identity in Theorem \ref{thm:id}.

Let $C$ be the genus $g$ hyperelliptic curve $y^2=x^{2g+1}+x$. The map 
$$\phi(x,y) = \left(\dfrac{x^2+1}{x}, \dfrac{y}{x^{a}} \right),$$
where $a=\frac{g+1}{2}$, is a nonconstant morphism from $C$ to some curve, denoted $C'$. Note that the curve $C'$ will also be hyperelliptic. We initially define $C'$ to be of the form
$$y^2=c_dx^d+\ldots +c_{d-i}x^{d-i}+\ldots + c_0$$
and we will apply the transformation of variables given by $\phi$ to determine the coefficients $c_j$. Applying the transformation yields
\begin{align*}\left(\frac{y}{x^a}\right)^2&=c_d\left(\dfrac{x^2+1}{x}\right)^d+\ldots+c_{d-i}\left(\dfrac{x^2+1}{x}\right)^{d-i}+\ldots + c_0\\
\frac{y^2}{x^{g+1}}&=c_dx^{-d}({x^2+1})^d+\ldots+c_{d-i}x^{i-d}({x^2+1})^{d-i}+\ldots + c_0\\
y^2&=c_dx^{g+1-d}({x^2+1})^d+\ldots+c_{d-i}x^{g+1+i-d}({x^2+1})^{d-i}+\ldots + c_0x^{g+1}.
\end{align*}
In order for $\phi$ to be a morphism from $C$ to $C'$, this last equation should, in fact, be the equation for the curve $C$. Note that the degree of the expression in $x$ will be $g+1-d+2d=g+1+d$. Hence, we need $c_d=1$ and $g+1+d = 2g+1$, so that $d=g$. We use this to simplify the above equation to
\begin{align}\label{eqn:generalC'}
y^2&=x({x^2+1})^g+\ldots +c_{g-i}x^{1+i}({x^2+1})^{g-i} +\ldots + c_0x^{g+1}.
\end{align}
In order to determine the coefficients $c_i$, we need to expand the right-hand side of the equation and match coefficients with those of $C$. We now work through two examples to better understand what the coefficients of $C'$ will be.

\begin{example}\label{ex:g=5}
Let $g=5$, so that $C$ is the hyperelliptic curve $y^2=x^{11}+x$. From our above work, we know that the degree of $C'$ will be $5$. Consider the following terms from Equation \ref{eqn:generalC'}: $A_1=x({x^2+1})^5$, $A_2=x^3({x^2+1})^3$, and $A_3=x^5({x^2+1})^1$. We expand each of these to get
\begin{align*}
    A_1&=x(x^{10}+5x^8+10x^6+10x^4+5x^2+1)\\
     &=x^{11}+5x^9+10x^7+10x^5+5x^3+x,\\
    A_2&=x^3(x^6+3x^4+3x^2+1)\\
     &=x^9+3x^7+3x^5+x^3,\\
    A_3&=x^5(x^2+1)\\
     &= x^7+x^5.
\end{align*}

Note that $A_1-5A_2+5A_3 = x^{11}+x$. Hence, $\phi$ is a morphism from $C$ to $ y^2=x^5-5x^3+5x$.
\end{example}

\begin{example}\label{ex:g=6}
Now let $g=6$, so that $C$ is the hyperelliptic curve $y^2=x^{13}+x$. From our above work, we know that the degree of $C'$ will be $6$. Consider the following terms from Equation \ref{eqn:generalC'}: $B_1=x({x^2+1})^6$, $B_2=x^3({x^2+1})^4$, $B_3=x^5({x^2+1})^2$, and $B_4=x^7(x^2+1)^0$. We expand each of these to get
\begin{align*}
    B_1&=x(x^{12}+6x^{10}+15x^8+2-x^6+15x^4+6x^2+1)\\
     &=x^{13}+6x^{11}+15x^9+2-x^7+15x^5+6x^3+x,\\
    B_2&=x^3(x^8+4x^6+6x^4+4x^2+1)\\
     &=x^{11}+4x^9+6x^7+4x^5+x^3,\\
    B_3&=x^5(x^4+2x^2+1)\\
     &= x^9+2x^7+x^5\\
    B_4&=x^7.
\end{align*}
One can easily show that $B_1-6B_2+9B_3-2B_4=x^{13}+x$, which tells us that $\phi$ is a morphism from $C$ to $y^2=x^6-6x^4+9x^2-2$.
\end{example}

While working on \cite[Section 5.1]{EmoryGoodsonPeyrot}, the author determined (by hand) the curve $C'$ for $g=11$. The coefficients she found were 1, 11, 44, 77, 55, and 11, with alternating signs (see Table \ref{tab:curvetable} below). The author entered this sequence of numbers into the On-line Encyclopedia of Integer Sequences \cite{oeis} and found that these numbers are the Triangle of coefficients of Lucas (or Cardan) polynomials, $T(n,k)$. The coefficients that appear in Examples \ref{ex:g=5} and \ref{ex:g=6} are also of the form $T(n,k)$. As noted in Remark \ref{rem:oeis}, $$T(n,k)= \frac{n}{n-k}\binom{n-k}{k}.$$
This leads us to the following theorems.

\begin{theorem}\label{thm:curvemorphism}
Let $C$ be the hyperelliptic curve $y^2=x^{2g+1}+x$ and let $C'$ be the hyperelliptic curve
$$y^2=\sum^{\lfloor g/2\rfloor}_{k=0} (-1)^{k}\frac{g}{g-k}\binom{g-k}{k} x^{g-2k}.$$
Then the map 
$$\phi(x,y) = \left(\dfrac{x^2+1}{x}, \dfrac{y}{x^{a}} \right),$$
where $a=\frac{g+1}{2}$, is a nonconstant morphism from $C$ to $C'$.
\end{theorem}

We can generalize Theorem \ref{thm:curvemorphism}. Let $c\in\mathbb Q^*$ be constant and $\zeta$ be a primitive $g$-th root of unity. In the following theorem we work over the field $\mathbb F=\mathbb Q(\zeta, c^{1/g})$.

\begin{theorem}\label{thm:curvemorphism2}
Let $C$ be the hyperelliptic curve $y^2=x^{2g+1}+cx$ and let $C_i$ be the hyperelliptic curve
$$y^2=\sum^{\lfloor g/2\rfloor}_{k=0} (-1)^{k}\frac{g}{g-k}\binom{g-k}{k}\zeta^{ik}c^{k/g} x^{g-2k}$$
for $i=0,1$. Then the map 
$$\phi_i(x,y) = \left(\dfrac{x^2+\zeta^{i}c^{1/g}}{x}, \dfrac{y}{x^{a}} \right),$$
where $a=\frac{g+1}{2}$, is a nonconstant morphism from $C$ to $C_i$.
\end{theorem}
Since 
$$\frac{g}{g-k}\binom{g-k}{k} = \left[\binom{g-k}{k}+\binom{g-k-1}{k-1}\right]$$
(see, for example, \cite[Section 9.9]{Koshy2014}), Theorem \ref{thm:curvemorphism2} also generalizes Lemma 5.1 in \cite{EmoryGoodsonPeyrot} because we are no longer restricting $g$ to be odd. The proofs of Theorems \ref{thm:curvemorphism} and \ref{thm:curvemorphism2} are nearly identical to the proof of Lemma 5.1 in \cite{EmoryGoodsonPeyrot}, and so we omit them.

We now expand Example \ref{ex:g=5} to show how our work in this section relates to our work in Theorem \ref{thm:id}. Let $n=g=5$, and $k$ range from 0 to $\lfloor n/2\rfloor=2$. We evaluate
$$(-1)^k\frac{n}{n-k}\binom{n-k}{k}$$
for each of these values of $k$ to get
\begin{align*}
    k=0 &:\; (-1)^0\frac{5}{5-0}\binom{5-0}{0}=1\\
    k=1 &:\; (-1)^1\frac{5}{5-1}\binom{5-1}{1}=-5\\
    k=2 &:\; (-1)^2\frac{5}{5-2}\binom{5-2}{2}=5,
\end{align*}
which are the coefficients in the equation for $C'$, i.e. those of $A_1,A_2,$ and $A_3$, respectively. These coefficients help us cancel certain powers of $x$ in the expansion of Equation \ref{eqn:generalC'}. For example, in the sum $A_1-5A_2+5A_3$, the coefficient of $x^5$ is
\begin{align*}
    0&=10 - 5\cdot 3 +5\cdot1\\
     &=\binom{5}{2}-5\binom{3}{1}+5\binom{1}{0}\\
     &=\frac{5}{5-0}\binom{5-0}{0}\binom{5}{2}-\frac{5}{5-1}\binom{5-1}{1}\binom{3}{1}+\frac{5}{5-2}\binom{5-2}{2}\binom{1}{0},
\end{align*}
which matches the statement of Theorem \ref{thm:id} for $n=5$ and $i=2$. 

\subsection{Higher Genus Examples}
Table \ref{tab:curvetable} below gives $C_i$ for values of $g$ up to 11 and for $c=1$. Note that this table expands on the table that appears in \cite[Section 5.1]{EmoryGoodsonPeyrot}.

\begin{table}[h]
\begin{tabular}{c|l}\setlength{\tabcolsep}{.5pt}
    $g$ \hspace{.1in}&\hspace{.1in} curve $C_i$\\
    \hline 
    5 &\hspace{.1in} $y^2=x^5-5\zeta^{i}x^3+5\zeta^{2i}x$\\
    6 &\hspace{.1in} $y^2=x^6-6\zeta^{i}x^4+9\zeta^{2i}x^2-2\zeta^{3i}$\\
    7 & \hspace{.1in} $y^2 = x^7-7\zeta^{i} x^5+14\zeta^{2i}x^3-7\zeta^{3i}x$\\
    8 &\hspace{.1in} $y^2 = x^8 -8\zeta^{i} x^6+20\zeta^{2i}x^4 - 16\zeta^{3i}x^2 +2 \zeta^{4i} $\\
    9 &\hspace{.1in} $y^2 = x^9 -9\zeta^{i} x^7+27\zeta^{2i}x^5 - 30\zeta^{3i}x^3 +9 \zeta^{4i} x$\\
    10 & \hspace{.1in}$y^2=x^{10}-10\zeta^{i}x^8+35\zeta^{2i}x^6-50\zeta^{3i}x^4+25\zeta^{4i}x^2-2\zeta^{5i}$\\
    11 & \hspace{.1in}$y^2=x^{11}-11\zeta^{i}x^9+44\zeta^{2i}x^7-77\zeta^{3i}x^5+55\zeta^{4i}x^3-11\zeta^{5i}x$
\end{tabular}
    \caption{ }
    \label{tab:curvetable}
\end{table}

Note that for all $g$, the coefficient of second term of the expression in $x$ will always be $-g$  (times a power of $\zeta$). The reason this is the case is that this coefficient corresponds to $k=1$, and
\begin{align*}
    (-1)^k\frac{g}{g-k}\binom{g-k}{k}&=-\frac{g}{g-1}\binom{g-1}{1}\\
                               &=-\frac{g}{g-1}\cdot(g-1)\\
                               &=-g.
\end{align*} 

Note that when $g$ is even, the final term corresponds to $k=g/2$, which yields $x^0$. We compute the coefficient to be 
\begin{align*}
    (-1)^k\frac{g}{g-k}\binom{g-k}{k}&=(-1)^{g/2}\frac{g}{g-g/2}\binom{g-g/2}{g/2}\\
    &=(-1)^{g/2}\frac{g}{g/2}\binom{g/2}{g/2}\\
                               &=(-1)^{g/2}2.
\end{align*} 
Hence, when $g$ is even, the final term of the expression in $x$ will always be $(-1)^{g/2}2$ (times a power of $\zeta$).

On the other hand, when $g$ is odd, the final term corresponds to $k=(g-1)/2$, which yields $x^1$. We compute the coefficient to be 
\begin{align*}
   (-1)^k \frac{g}{g-k}\binom{g-k}{k}&=(-1)^{(g-1)/2}\frac{g}{g-(g-1)/2}\binom{g-(g-1)/2}{(g-1)/2}\\
                               &=(-1)^{(g-1)/2}\frac{g}{(g+1)/2}\binom{(g+1)/2}{(g-1)/2}\\
                               &=(-1)^{(g-1)/2}\frac{g}{(g+1)/2}\binom{(g-1)/2+ 1}{(g-1)/2}\\
                               &=(-1)^{(g-1)/2}\frac{g}{(g+1)/2}\cdot ((g-1)/2+ 1)\\
                               &=(-1)^{(g-1)/2}g.
\end{align*} 

Hence, when $g$ is odd, the final term of the expression in $x$ will always be $(-1)^{(g-1)/2}gx$ (times a power of $\zeta$).

\section*{Acknowledgements}
The author thanks Darij Grinberg for helpful comments on an earlier draft of this paper.

\end{document}